\documentclass[12pt]{amsart}
\usepackage{amsfonts}
\usepackage{amscd}
\usepackage{amssymb}
\usepackage{graphics}
\usepackage{enumerate}
\usepackage{amsmath}

\usepackage{hyperref}
\hypersetup{colorlinks,citecolor=blue}

\newcommand{\BZ}{{\mathbb{Z}}}
\newcommand{\BN}{{\mathbb{N}}}
\newcommand{\BR}{{\mathbb{R}}}
\newcommand{\BC}{{\mathbb{C}}}
\newcommand{\BF}{{\mathbb{F}}}

\newcommand{\BE}{{\mathbb{E}}}

\newcommand{\gD}{\Delta}
\newcommand{\gd}{\delta}

\newcommand{\gC}{\Gamma}

\newcommand{\gl}{\lambda}
\newcommand{\ga}{\alpha}

\newcommand{\ti}[1]{\tilde{#1}}

\newtheorem{prop}{Proposition}[section]
\newtheorem{thm}[prop]{Theorem}
\newtheorem{lem}[prop]{Lemma}
\newtheorem{cor}[prop]{Corollary}

\theoremstyle{definition}

\newtheorem{defn}[prop]{Definition}
\newtheorem{rem}[prop]{Remark}

\newtheorem{exam}[prop]{Example}

\newtheorem{clm}[prop]{Claim}

\newcommand\Isom{\operatorname{Isom}}

\catcode`\@=11
\long\def\@savemarbox#1#2{\global\setbox#1\vtop{\hsize\marginparwidth
  \@parboxrestore\tiny\raggedright #2}}
\catcode`\@=12
\numberwithin{equation}{section}

\begin{document}
\author{Omer Lavy}

\thanks{}

\date{17.12.15}

\title{fixed point theorems for groups acting on non-positively curved manifolds }
\maketitle

\begin{abstract}
We study isometric actions of Steinberg groups on Hadamard manifolds. We prove some rigidity properties related to these actions. In particular we show that every isometric action of $St_n(\BF_p\langle t_1,\ldots ,t_k \rangle)$ on Hadamard manifold   when $n \geq 3$ factors through a finite quotient. We further study actions on infinite dimensional manifolds and prove a fixed point theorem related to such actions.
\end{abstract}

\section{Introduction}
We study isometric actions of non-commutative Steinberg groups on Hadamard manifolds. Hadamard manifolds are complete simply connected non-positively curved Riemaniann manifolds. Usually Hadamard manifolds are assumed to be of finite dimension. We also consider the infinite dimension case. Recall that while finite dimensional manifolds are metrically proper (i.e. closed balls are compact), infinite dimensional manifolds are not hence we will have different treatment for each case.

  It is a well known question of Gromov whether there exist groups with no fixed point free action on CAT(0) spaces. Gromov conjectured that random groups always have a fixed point (see Pansu \cite {P}.) A first step in this direction was done by Arzhantseva et al.. They introduced an example of an infinite group that admits no non-trivial isometric action on finite dimensional manifolds which are p-acyclic \cite {ABJLMS}. Next it was shown by Naor and Silberman \cite {NS} that indeed not only that random groups have fixed points when acting on CAT(0) spaces, but that this property can be extended to many p-convex metric spaces.

  We focus our attention on the higher rank Steinberg  groups, $St_n(R)$ when $n \geq 3$ and $R$ is either the associative ring $R=\BF_p\langle t_1,\ldots ,t_k \rangle$ (for some applications we require that $p \geq 5$) or the torsion free ring $R=\BZ \langle t_1, \ldots ,t_k\rangle$ (we use the $\langle \rangle$ sign to denote non-commutative polynomials). These groups are often called non-commutative universal lattices. Kassabov (and Shalom in the commutative case) coined the name as they surject on many lattices in higher rank Lie groups. It is for this reason that any fixed point property proved for them immediately applies to the corresponding lattices. Since lattices in p-adic analytic groups and in Lie groups do have fixed point free actions on CAT(0) spaces (their associated buildings and symmetric spaces for example) one can not hope to have such a strong result concerning their actions. We have therefore to assume more. 

 Our first goal is to study isometric actions of $St_n(\BF_p \langle t_1, \ldots , t_k \rangle)$ on (finite dimensional) Hadamard manifolds. We show that any isometric action of the groups $\gC=St_n(R)$   when $R=\BF_p \langle t_1, \ldots , t_k \rangle$ ($n \geq 3$) on a finite dimensional Hadamard manifold is finite.

\begin {thm} \label {main result}
Let $\gC=St_n(R)$   when $R=\BF_p \langle t_1, \ldots , t_k \rangle$. Then any isometric group action of $\gC$ on a finite dimensional Hadamard manifold $X$ is finite, i.e. it factors through a finite group (in particular $\gC$ has a fixed point in $X$.)
\end {thm}

\begin {rem}
Note that the (infinite dimensional) regular representation $\gC \rightarrow U(l^2(\gC))$ is a $\gC$ isometric action which is not finite.
\end {rem}

 When the dimension of $X$ is infinite it is not proper anymore and more delicate methods are needed. We restrict our treatment to pinched manifolds. These are manifolds whose sectional curvature is bounded from below as well. We show that this is enough to ensure that $\gC$ has a fixed point in $X$, provided that $p\ge 5$.
\begin {thm} \label {infinite dimension}
 Let $\gC$ be as above with $p \ge 5$. Let $X$ be an Hadamard manifold. If the sectional curvature of $X$ is bounded from below ($X$ can be of infinite dimension here) then $\gC$ has a global fixed point in $X$.
\end {thm}

For the Steinberg groups defined over the ring $R=\BZ \langle t_1, \ldots ,t_k\rangle$ such a theorem cannot be true. Being an unbounded subgroup in $SL_n(\BR)$, $SL_n(\BZ)$ is acting on the symmetric space associated with $SL_n(\BR)$ without a fixed point. Since $SL_n(\BZ)$ is a quotient of $\gC$ this induces a fixed point free $\gC$ isometric action. However the following is true.
 \begin {thm} \label {non  commutative universal lattices}
 Let $\gC=St_n(R)$ $(n \geq 3)$ when $R=\BZ \langle   t_1,\ldots,t_k\rangle$. Suppose that $X$ is a CAT(0) space and that $H$ is a group acting by isometries on $X$ properly and co-compactly (and faithfully) then any homomorphism $\phi : \gC \rightarrow H$ has a finite image.
 \end {thm}
\begin {rem}
Recall that $SL_n(\BZ)$ is a non-uniform lattice in $SL_n(\BR)$. The theorem above gives a nice rigidity property. Namely, $SL_n(\BZ)$ cannot be mapped onto co-compact lattices in CAT(0) groups.
\end{rem}
\begin {rem}
We point out that a fixed point theorem for these groups acting on low dimensional CAT(0) cell complexes was established by Farb (see \cite {F}.)
\end{rem}

\subsection {Ideas and Techniques}
 Results similar to that of Theorem \ref {main result} were obtained by Wang, followed by the work of Izeki and Nayatani (see \cite {W}, \cite {IN}) who showed that many lattices in semi-simple algebraic groups over p-adic field have fixed point property. As mentioned above Naor and Silberman also obtained fixed point property related to action of random groups on many convex spaces. In both cases the results were obtained by carrying some averaging process. This process yields some heat equations. Spectral gap ensures the process terminates with a fixed point. A key step is to obtain Poincare inequalities. Those are in general hard to obtain. The methods just described are inspired by Zuk's criteria used for proving property (T). Our techniques are also borrowed from methods used for proving property (T). We try to adopt the geometric approach.

 The geometric approach towards proving property (T) was first introduced by Dymara and Januszkiewicz in \cite {DJ}, and then developed by Ershov, Jaikin and Kassabov \cite {EJ} \cite {K} \cite {EJK}. The main idea is to examine angles between invariant spaces of finite (compact in the non-discrete case) subgroups generating $\gC$. Since these groups are finite, each of them has property (T) which means almost invariant vectors are "close" to invariant vectors. If on the other hand the angles between any two respective invariant vector spaces is "large enough" then the invariant vectors spaces of the finite groups are "far" from each other. The conclusion is that when no non-zero $\gC$ invariant vectors exist, almost invariant vectors are trivial and the group has property (T). In the case $\gC=\langle G_1,G_2,G_3\rangle $ the meaning of "large enough" is that these angles' sum is greater than $\pi$ (see \cite {EJ} and \cite {K}).

 Our method is similar. We study the action of small (finite mostly) subgroups of $\gC$. We use the relations between these small groups to study the action of the bigger group. When proving fixed point property for Hadamard manifolds we seek "fat" triangles. By saying "fat" we mean triangles in which the sum of the angles is greater than $\pi$. We will present a triangle whose vertices are fixed by the finite groups and that the angle between any two sides of it is at least the angle between the invariant spaces. In our case, we look at triangle which is minimal in the sense that the sum of squares of lengths of its sides is minimal. As the sum of the angles in any CAT(0) space can't be larger than $\pi$ we deduce that the triangle is a single point.
Recently (and independently) Ershov and Jaikin adopted a similar method and proved a fixed point theorem regarding  isometric group actions of these groups on $L_p$ spaces.  Mimura \cite {Mim} used different (purely algebraic) methods and proved fixed point properties related also to non commutative $L_p$ spaces (provided that $n \geq 4$.)

\subsection {Property FH}
When the underlying space is a Hilbert space $\mathcal H$ the ideas introduced above become very explicit. In this section we illustrate these ideas by giving an affine version of Kassabov's proof for the fact that these groups have property (T) (compare with Theorem 5.9 in \cite {EJ} and Theorem 1.2 in \cite {K}). We prove :
\begin {thm} \label {FH}
Let $G$ be a group satisfying the following properties:
\begin {enumerate}
\item $G= \langle G_1,G_2,G_3 \rangle$ where each pair $G_i,G_j$ generates a finite group.
\item For any orthogonal representation $(\pi,\mathcal H) $ of the groups $G_{i,j}=\langle {G_i, G_j}\rangle$, every $v \in \mathcal H$ satisfies the following property:
\begin {equation} \label {bbc} 
 d^2_0(v) < 2(d^2_i(v) +d_j^2(v))
\end {equation}
where $d_0(v)$ denotes the distance of $v$ from $ \mathcal H^{G_{i,j}}$, the (closed) space of $G_{i,j}$ invariant vectors, and $d_i(v)$ measure the distance between $v$ and $\mathcal H^{G_i}$.
\end {enumerate}
then $G$ has property FH.
\end {thm}

\begin {rem}
\begin{enumerate}
\item It is readily verified that \ref {bbc} is equivalent to having angles greater than $\pi/3$ between the corresponding subgroups as defined in the next section (see discussion in \cite {K}). This together with \ref {lem:>60} give the desired result regarding the Steinberg groups.
\item The fact that $G_{i,j}$ are finite ensures that $ \mathcal H^{G_{i,j}}$ is not empty in any $G$ isometric affine action.
\end{enumerate}
\end {rem}
As explained above we are interested in fat triangles. We will introduce one by minimizing the radius of the barycentric circle. Given an affine isometric $G$ action, $(\rho,\mathcal H)$ we define a function $f:\mathcal H \rightarrow \BR$, by
$$x \mapsto d^2(x,\mathcal H^{G_{1,2}})+d^2(x,\mathcal H^{G_{1,3}})+d^2(x,\mathcal H^{G_{2,3}}).$$
\begin {clm}
Suppose that $(\rho,\mathcal H)$ is an isometric affine action and $f$ is the function defined above then $f$ attains a minimum.
\end {clm}

\begin {proof}
Indeed the affine map $x \mapsto ((x-\pi_{1,2}(x)),(x-\pi_{1,3}(x)),(x-\pi_{2,3}(x)))$ (with $\pi_{i,j}$ denoting the projection on $\mathcal H^{G_{i,j}}$) maps $\mathcal H$ onto an affine subspace of $\mathcal H \times \mathcal H \times \mathcal H$. A pre-image of the closest point to $0$ in this subspace is a minimum.
\end {proof}

\begin {proof}[Proof of Theorem \ref {FH}]
Suppose towards contradiction that $(\rho ,\mathcal H)$ is an affine isometric fixed point free $G$ action. Let $q \in \mathcal H$ be a point minimizing $f$. For simplicity denote the projections of $q$ on the fixed points spaces $\mathcal H^{G_{i,j}}$ by $x,y,z$. Note that since $q$ is minimizing for $f$ we can assume that it is the barycenter of $\{x,y,z\}$ this means $q=\frac {x+y+z} {3}$. The restriction of the action to the corresponding subgroup is equivalent to an orthogonal representation. By applying \ref {bbc} three times and summing we obtain:
\[
d^2(q,x)+d^2(q,y)+d^2(q,x) < 4(d_1^2(q)+d_2^2(q)+d_3^2(q)) \leq 4(d^2(q,[x,y])+d^2(q,[x,z])+d^2(q,[y,z]))
\]
(the second inequality follows from the fact that the segment connecting two vertices is fixed by the intersection of the corresponding subgroups). However for a barycenter point in an Euclidean triangle this is impossible. Indeed it is well known that the barycenter lies on the intersection of the medians. The barycenter divides each median segment into two subsegments. The first connects the barycenter to the vertex and is twice as long as the second which connects the barycenter to the middle of the opposite side. In general the segment connecting the barycenter to the middle of the side is longer than the distance from the barycenter to that side.
We have then:
\[
d^2(q,x) \geq 4d^2(q,[y,z]).
\]
(and same for the other vertices.) This gives a contradiction and the statement is proved.
\end {proof}

\subsection *{Acknowledgment}
Many ideas appearing in this paper were suggested by T. Gelander and U. Bader. The question concerning finiteness of the Steinberg groups over finite index ideals was addressed to M. Ershov and I. Rapinchuk who both gave quick and helpful responses.


\section{preliminaries}
\subsection {Angles between  Invariant Subspaces}
Let $H$ be a finite group acting on a Hadamard manifold $X$. Recall that Hadamard manifolds are complete simply connected non-positively curved Riemaniann manifolds (possibly of infinite dimension). By a classical theorem of Cartan $H$ fixes a point in $X$. Suppose that $x_0 \in X$ is fixed by $H$ and that $\xi$ is a geodesic ray issuing from $x_0$, then $\xi$ is mapped onto another ray also issuing from $x_0$. The action then reduces to a representation on the tangent space at $x_0$, denoted by $T_{x_0}$. Furthermore as isometric maps preserves angles, this representation is actually orthogonal. This motivates the study of angles between invariant subspaces in orthogonal representations in the context of isometric actions on manifolds.
Recall Kassabov's definition for angles between closed subspaces (see \cite {K}):
\begin {defn} \label {defn: angles between subspaces}
Let $V_1$ , $V_2$ be two closed subspaces in a Hilbert space. We define the angle between $V_1$ and $V_2$ to be the infimum over the angles between vectors $v_i \in V_i$  $i=1,2$ such that $v_i \perp V_1 \cap V_2$ .
i.e. $$\sphericalangle (V_1 ,V_2) = inf \{ \sphericalangle (v_1,v_2)\mid v_i\in V_i ~and~ v_i \perp V_1 \cap V_2 \}$$
\end {defn}
Note that this is equivalent to say that $$ cos (\sphericalangle (V_1 ,V_2)) = sup \left\{\frac {\langle v_1,v_2 \rangle} {\| v_1\| \| v_2\| }\mid v_i\in V_i ~and~ v_i \perp V_1 \cap V_2 \right\}$$

When $V$ is a unitary representation of $G$ we denote by $V^G$ the (closed) subspace of invariant vectors in $V$. It is convenient then to define angle between subgroups:
\begin {defn} \label {angles between subgroups}

Suppose $H= \langle G_1,G_2 \rangle$ the angle between $G_1$ and $G_2$ is defined as
$$\sphericalangle (G_1,G_2)= inf \{\sphericalangle (V^{G_1} ,V^{G_2} ) \mid V \mbox{ is a unitary representation   of  H}\}$$
\end {defn}
\begin {rem}  \label {angles in real vector space}
\begin {enumerate} 
\item The tangent space at a point $x \in X$ is a real vector space. An isometric representation on a real vector space will be called Orthogonal while an isometric representation on a complex vector space will be called Unitary.
\item Given an orthogonal representation on a real vector space $V$, denote by $U=V^{\BC}$ the complexification of $V$, $U=V \otimes \BC$. Given subgroups $G_1,~G_2$ and an orthogonal representation on a real vector space $V$, one can easily verify that $\sphericalangle (V^{G_1} ,V^{G_2}) =\sphericalangle (U^{ G_1} ,U^{ G_2}) $. Therefore  $\sphericalangle (G_1,G_2)$ forms a lower bound on angles between invariant subspaces in real vector space.
\end {enumerate}
\end {rem}
\begin {rem}
Recall that a representation of a finite (compact) group can be decomposed as a direct sum of irreducible ones. Thus, when $G_{1,2}$ is finite the phrase "any unitary representation" in the Definition \ref {angles between subgroups} is equivalent  to "any irreducible unitary representation".
\end {rem}
In the next section we are going to give a criterion for a group, generated by finite subgroups, to have fixed point property.

\begin {thm}  \label {infinite fix point criterion}
Let $G=\langle G_1,G_2,G_3\rangle$ where $G_{i,j}=\langle G_i,G_j\rangle~i,j=1,2,3$ are finite groups. Suppose that $G$ is acting isometrically on a Hadamard manifold $X$.
If the sectional curvature of $X$ is bounded from below ($X$ can be of infinite dimension then) and  there exist $\theta > \pi/3$ such that $\sphericalangle (G_i,G_j) \ge \theta,~1 \leq i \ne j \le 3$  then $G$ has a global fixed point in $X$.
\end {thm} 
In order to apply Theorems \ref {infinite fix point criterion} we need to study representation theory of finite subgroups of the Steinberg group.

\subsection {The Steinberg Group Over a Unital Ring}
Recall the definition of the Steinberg group over unital ring. Let $R$ be any unital ring (in our case $R$ will be $\BZ \langle   t_1,\ldots,t_k\rangle$ or $\BF_p \langle   t_1,\ldots,t_k\rangle$ ). The Steinberg group over $R$ of dimension n, denoted $St_n(R)$, is defined to be the group generated by $x_{i,j}(r)$ where $r \in R$ and $ 1 \leq i \neq j \leq n$, subject to the relations:
\begin {enumerate}
\item $x_{i,j}(r_1)x_{i,j}(r_2)=x_{i,j}(r_1+r_2)$
\item $[x_{i,j}(r_1),x_{j,k}(r_2)]=x_{i,k}(r_1r_2)$
\item $[x_{i,j}(r_1),x_{l,k}(r_2)]=1$ when $j\ne l$.
\end {enumerate}
Where $[x,y]=x^{-1}y^{-1}xy$.

\begin {rem}
\begin {enumerate}
\item The map defined by $x_{i,j}(r) \mapsto e_{i,j}(r)$ ($e_{i,j}(r)$ is the elementary matrix with $1$ on the diagonal, $r$ in the $(i,j)$ place and $0$ elsewhere) can be extended to a surjection: $\phi :St_n(R) \rightarrow EL_n(R)$ to the normal group generated by elementary matrices. If $R$ is commutative there is a natural definition of determinant and $EL_n(R)$ is a subgroup of $SL_n(R)$ (the kernel of the determinant map).
\item This is related to Algebraic K-Theory. The quotient 
$$SL_n(R)/EL_n(R)$$
is denoted as $SK_1(n,R)$. Further the kernel of $\phi$ is closely related to $K_2(R)$ (it is a subgroup of $K_2(R)$).
\end {enumerate}
\end{rem}
\begin {exam}
When $R=\BZ$ and $n \ge3$ this becomes very explicit:
\begin {enumerate}
\item It is easy to verify that any matrix in $SL_n(\BZ)$ can be written as a product of elements of $EL_n(\BZ)$ hence $SK_1(n,\BZ)$ is trivial.
\item $K_2$ however is not trivial. For example 
$$x=(x_{1,2}(1)x_{2,1}(-1)x_{1,2})^4$$
is an element of order $2$ in the kernel of $\phi$. It is true however that the kernel of $\phi$ has exactly two elements (for this see Theorem 10.1 of \cite {M}). This gives an alternative description of $SL_n(\BZ)$ in terms of generators and relations.
\end {enumerate}
\end {exam}

 Next we collect some basic facts regarding to $St_n(R)$ and its representation theory. Throughout assume $R=\BF_p \langle   t_1,\ldots,t_k\rangle$ (similar results are true for $R=\BZ \langle   t_1,\ldots,t_k\rangle$). The following claim is easily verified:
\begin {clm}
The group $St_n(R)$ is generated by the following subgroups:

\begin{itemize}
\item
$G_1= \langle x_{1,n-1}(1),x_{2,n-1}(1), \ldots , x_{n-2,n-1}(1)\rangle \cong \BF_p^{n-2} $.
\item $G_2=\langle x_{n-1,n}(1)\rangle \cong \BF_p$ .
\item $G_3=\langle x_{n,1}(a_0+a_1t_1+\ldots a_kt_k) ,\ldots x_{n,n-2}(a_0+a_1t_1+\ldots a_kt_k)\rangle \cong (\BF_p^{n-2})^k$.

\end{itemize} 
\end {clm}
\begin {rem}
It is easily verified that for any $1 \le i\ne j \le 3$ the groups $G_{i,j}=\langle G_i,G_j \rangle$ are finite. (see Section 4.1 in \cite {EJ} for more details.)
\end {rem}
We are interested in the angles between $G_i$ and $G_j$.
\begin {lem} \label {lem:>60}
Suppose $p \geq 5$, then there exist $\gd>0$ such that for every irreducible unitary representation $(\pi,V)$ of $\langle G_i,G_j\rangle,~1\le i\ne j\le 3$ the angle $\sphericalangle {(V^{\pi(G_i)},V^{\pi(G_j)})}> \pi/3+\gd$. In particular $\sphericalangle{(G_i,G_j)}> \pi /3$
\end{lem} 
 (see Section 4.1 in \cite {EJ})
The ideas behind the proof are illustrated in the next example:
\begin {exam}
Assume $R=\BF_p \langle   t_1,\ldots,t_k\rangle $ and that $n=3$. The subgroup $\langle   G_1,G_2\rangle$ is isomorphic to the (order $p^3$) Heisenberg group $H_p$ over
$\BF_p$. The Heisenberg group, $H_p$ is generated by
\[ x= \left( \begin{array} {ccc}
1 &  1 & 0 \\
0 & 1 & 0 \\
0 & 0 & 1 \\
\end{array} \right) , y= \left( \begin{array} {ccc}
1 &  0 & 0 \\
0 & 1 & 1 \\
0 & 0 & 1 \\
\end{array} \right) , z= \left( \begin{array} {ccc}
1 &  0 & 1 \\
0 & 1 & 0 \\
0 & 0 & 1 \\
\end{array} \right)\]
with the obvious identifications $G_1\cong\ti G_1=\langle  x\rangle$ and $G_2\cong \ti G_2=\langle  y\rangle$.
Further note that $H_p$ can be decomposed as a semi direct product: $H_p= \ti G_1 \ltimes A$ where $A \cong \BF_p^2 $ is the abelian group generated by $z$ and $y$ and $\ti{G_1}\cong \Bigg \{\begin {pmatrix} 1 & x \\ 0 & 1 \end {pmatrix} | x \in \BF_p \Bigg \}$(with this identification, the action of $\ti G_1$ on $A$ is via left matrix multiplication).  A complete description of the irreducible representations of this group is given in \cite {EJ}. Given a unitary irreducible representation of $H_p$ $(\pi,V)$,  its restriction to $A$ decomposes as a direct sum of characters  upon which $\ti G_1$ acts (identifying the dual of $A$ with itself the action is by inverse transpose multiplication). Given a character $\chi$ its orbit may have either $p$ elements or it is fixed by $\ti G_1$ (since $\ti G_1$ is of order $p$). In the former case one obtains a $p$ dimensional space. In the latter case, the center of $H_p$ (which is the group generated by $z$) is acting trivially. In this case $\pi(x)$ intertwines the action of $A$ hence by Schur's lemma the restriction of $\pi$ to $\ti G_1$ is a character.  The representation of $H_p$ is then a character factoring through the abelianzation of $H_p$, $H_p/{\langle z \rangle}$ (which is homomorphic to $\BF_p^2$). So far we found $p^2$ representations of dimension 1 and $p-1$ of dimension $p$. By counting we observe that we found all. Let us describe the latter more detailed: let $e_1,\ldots,e_p$ be the natural basis of $\BC^p$ and let $\eta $ be a non-trivial $p$'th root of unity. Define

\[
 \pi(x)e_i=e_{i+1}    ~\text{(cyclic) and}~ \pi(y)e_i=\eta^{i-1}e_i.
\]
In this case the spaces of invariant vectors are: $H^{\ti G_1}=\BC(e_1+\ldots +e_p)$ and $H^{\ti G_2}=\BC e_1$ and $\cos (\sphericalangle (H^{\ti G_1},H^{\ti G_2} )) = {1\over \sqrt p}$. We conclude then that if $p \ge 5$ then $\sphericalangle{(G_1,G_2)}> \pi /3$. The proof for $\sphericalangle{(G_1,G_2)}$ and $\sphericalangle{(G_1,G_3)}$ follows the same line.
\end {exam}

\subsection {Ultra-Products}
Next we recall the construction of ultra-products of Hadamard manifolds. Limits of metric spaces can be a powerful tool. In our case we will refine the metric in a given manifold. We will assume that the group is acting fixed point freely and use this assumption in order to construct a sequence of marked manifolds that become more and more flat. By taking a limit we obtain a Hilbert space upon which the group is acting without a fixed point.

In general, a sequence of metric spaces does not necessarily has a convergence subsequence.  A nice way to overcome this problem is by passing to ultra-limits. A more complete description of ultra-limits of metric spaces can be found in chapter I.5 in \cite {BH}.

 Let $(X_n,x_n)$ be a sequence of marked Hadamard manifolds. Fix a non-principal ultra-filter $\mathcal {U}$ on $\BN$. The ultra-product of $(X_n,x_n)$ with respect to $\mathcal{U}$, denoted by $(X_n,x_n)_{\mathcal{U}}$ is the quotient:
\[ 
(X_n,x_n)_{\mathcal{U}}= \Big( \prod_n (X_n,x_n) \Big)_{\infty}/ \mathcal{N}
\]
where
\[
\Big( \prod_n (X_n,x_n) \Big)_{\infty}=\{(y_n) | y_ n \in X_n,\: \sup_{ n} \: d(x_n,y_n)<\infty\}
\]
and $\mathcal {N}$ is an equivalence relation identifying sequences of zero distance:
\[
\mathcal {N} = \{ (y,z) \in \Big (\prod_n (X_n,x_n) \Big)^2_{\infty} | \lim_{U}{d(y_n,z_n)}=0\}
\]
 Suppose that $\alpha_n:G \rightarrow \Isom(X_n)$ are group actions on $X_n$. If for every group element $g \in G$, and every $y=(y)_n \in (X_n,x_n)_{\mathcal{U}}$, the sequence $d(\alpha_n (g)y_n,x_n)$ is bounded (actually it is enough to assume this for $\ga_n(g)x_n$), the following formula is well defined and produces an isometric action on the limit space.

\begin {equation} \label {action on product}
\alpha(g)(y)= (\alpha_n(g)y_n)
\end {equation}

\begin {exam} \label {limit of manifolds}
\begin {enumerate} [I.]
\item  An ultra-limit of geodesically complete spaces is also geodesically complete. An ultra-limit of complete spaces is also complete (see \cite {BH}.)
\item Ultra-limit of CAT(0) spaces is also CAT(0) space. Indeed CAT(0) spaces are characterized by the property that for every triple of points $x,y,z$ the following inequality holds:
\[
d^2(x,m(y,z)) \leq \frac{1}{2}d^2(x,y)+\frac{1}{2}d^2(x,z)-\frac{1}{4}d^2(y,z)
\]
(where $m(y,z)$ is the midpoint between $y$ and $z$.)
Note that in  inner product spaces this is an equality. Moreover, complete geodesically complete, spaces for which this is equality are Hilbert spaces. This motivates the following example. 
\item Suppose that $X$ is an infinite dimensional Hadamard manifold whose sectional curvature is bounded from below, and that $\{x_n\}$ is any sequence in $X$. Suppose further that $\{\gl_n\}$ is a sequence of real numbers with $\lim_{n \rightarrow \infty} {\gl_n}=\infty$ then the ultralimt of $(\gl_nX,x_n)$ is a Hilbert space (where $\gl_n X$ is the space $X$ whose metric $d$ is multiplied by $\gl_n$.) 
\end {enumerate}
 
\end{exam}

\section {subspace arrangements and fixed point property }
We now begin with some useful facts to be used in the proof of Theorem \ref {infinite fix point criterion}. Throughout this section we assume that $G$ is a group generated by finite groups, $G=\langle G_1,G_2,G_3\rangle$. We further assume that any pair $G_i,G_j$ generates a finite group. The main idea is to find "fat" triangles whose vertices are fixed by the action restricted to $G_{i,j}=\langle G_i,G_j \rangle$. By assumption $ G_{i,j}$ are finite. Hence by Cartan's theorem (see II.2.7 in \cite {BH}) they have fixed points.

 Suppose that $H$ is a finite group acting on a Hadamard manifold $X$.  We denote by $X^H$ the set of $H$ fixed points in $X$. Note that when $X$ is a Riemmanian manifold $X^H$ is a closed submanifold.
 When $X$ is a Riemmanian manifold and $x \in X$ is fixed by a group $H$ we can treat the action of $H$ as an orthogonal representation on the tangent space $T_x$. We wish to understand triangles whose vertices lie in $X^{G_{i,j}}$. We will do this in several steps. Recall that if $\{X_i\}_{i \in I}$ is a family of complete CAT(0) spaces, their product $\prod_{i\in I} X_i$ with the $L_2$ metric is also a complete CAT(0) space . Let $T$ be the space of triangles:
 $$ T=X^{G_{1,2}} \times X^{G_{1,3}} \times X^{G_{2,3}},$$
 and define
$$ f:T \rightarrow  \BR^+, ~ (x,y,z) \mapsto d^2(x,y)+d^2(z,y)+d^2(z,x)$$
for $x \in X^{G_{1,2}},~ y \in X^{G_{1,3}},~ z \in X^{G_{2,3}}$.
\begin {rem}
One can define also $f^1$ as $f^1(x,y,z)=d(x,y)+d(z,y)+d(z,x)$. Note that $f=0$ iff $f^1 =0$ and also $\inf{f}=0$ iff $\inf {f^1}=0$. The advantage of defining $f$ the way we did is that if we have $f \rightarrow \infty$ then $f$ has unique minimum while $f^1 $ has a minimum which is not necessarily unique. On the other hand calculations with $f^1$ are often easier.
\end {rem}

We claim that a minimal triangle is "fat" i.e. the sum of its angles is greater than $\pi$. This will play a significant role in the proof of Theorem \ref {infinite fix point criterion} as the sum of angles in a triangle in CAT(0) space can't be greater than $\pi$. This should follow from our assumption on the angles between invariant subspaces in orthogonal representations. Indeed since we have fixed points, the restrictions of the action to the finite subgroups are orthogonal representations. This suggests that the angles between invariant submanifolds should also have sum which is greater than $\pi$. The problem is that our definition of angles "mods out" the intersection between the invariant subspaces. A geodesic path connecting say the vertex $x$  to $y$ however, does not necessarily have derivatives perpendicular to $T_x^{\langle G_1,G_2\rangle}$. The next claim deals with this matter:
\begin {clm} \label {angles in minimal triangle}
Let $x\in X^{\langle G_i,G_j\rangle}$ be a vertex in a minimal triangle as above, and let $c_1(t)\subset X^{G_i},c_2(t)\subset X^{G_j}$ be the geodesic paths issuing from $x$ to $y$ and $z$ respectively then $\sphericalangle (c'_1(0),c'_2(0))\geq \sphericalangle(T_x^{G_i},T_x^{G_j})$
\end {clm}
\begin {proof}
 Suppose that $\sphericalangle (c'_1(0),c'_2(0)) < \sphericalangle(T_x^{G_i},T_x^{G_j})$. For convenience denote $V=c'_1(0)$ and $W=c'_2(0)$, also write $V=V_0+V^\perp$ where $V_0=P_{i,j} V$ is the orthogonal projection of $V$ on   $T_x^{\langle G_i,G_j\rangle}$  (use the same notation for $W$). We will show that for some $x' \in X^{\langle G_i,G_j\rangle}$ we get $f(x',y,z)<f(x,y,z)$. Since $\sphericalangle (V,W) < \sphericalangle(T_x^{G_i},T_x^{G_j})$ we have that $\langle V_0, W_0 \rangle _{T_x} >0$. This means that the angle between $V$ and $W_0$ in $T_x$ is acute. Now denote by $w=w(t)$ the exponent of $W_0$ in $X^{\langle G_i,G_j\rangle}$. Since $\sphericalangle (w(t),c_1)<\pi/2$ and $\sphericalangle (w(t),c_2)<\pi/2$ we have that for some (every) $t_0$ small enough there exist $y' \in c_1$ and $z' \in c_2$ for which in the comparison triangles $\bar{\Delta} (\bar{x},\overline{w(t_0)},\overline {y'})$, and $\bar{\Delta} (\bar{x},\overline{w(t_0)},\overline {z'})$, the angles at $\bar x$ will also be  smaller than $\pi/2$. This together with the CAT(0) inequality, would imply that for some (any close enough) point $p \in [x,w(t_0)]$ we would have $d(p,z')<d(x,z')$ and also $d(p,y')<d(x,y')$. By the triangle inequality, $d(p,z)\leq d(p,z')+d(z',z).$ We get then that $d(p,z)<d(x,z')+d(z',z)=d(x,z)$  and similarly, also $d(p,y)<d(x,y)$ hence also $f(p,y,z)<f(x,y,z)$. 
\end{proof}
\begin {rem}
One can use the claim above and prove a finite dimensional version of \ref {infinite fix point criterion}. More precisely, denote by $ \partial X$ the visual boundary of $X$. We write then $\ti X =X \cup \partial X$. When $X$ is a finite dimensional manifold, $\partial X$ can be thought of as a sphere at infinity and $\ti X$ is then a compact space. Any isometric action on $X$ can be extended to an action on $\ti X$ (for all this see section II.8 in \cite{BH}). We claim that whenever such a group acts on a finite dimensional Hadamard manifold $X$, fixed point freely, it must fix a  point at infinity (i.e. it must have a fixed point at $\partial X$). Indeed the function $f$ defined above is convex hence if $f \rightarrow \infty$ as $x \rightarrow \infty$ , it has a minimum (see for example \cite {GKM}). That minimum is by the claim above a fixed point.\\
 On the other hand if $f$ does not tend to infinity as $x$ does, then by compactness of  $\tilde X$ it has a fixed point at infinity
\end {rem}
Claim \ref {angles in minimal triangle} is the the first step in proving  Theorem \ref {infinite fix point criterion}. Note that existence of fixed point is equivalent to having $\min_{T}{f} = 0$. We will show next that $\inf_{T}{f} =0$. The proof of this requires some preparation which will be done in the next two claims and the following remark.
%
\begin {clm}\label{switching vertices}
Suppose that $f(x,y,z) > L>0$. Suppose further that $d(x,y)< \sqrt{L/10}$, then the angle at $z$ is smaller than $\pi/3$.
\end{clm}

\begin {proof}
Indeed if $d(x,y)<\sqrt{ L/10}$ then either $d(x,z)> 3\sqrt{L/10}$, or $d(y,z)> 3\sqrt{L/10}$. Denote by $r$,  the diameter of the circumcircle of the comparison triangle at $\BE^2$. Then $r> d(y,z)> 3\sqrt{L/10}$. Then by the law of sines we get that  the angle at $z$ is smaller than $\pi/3$.
\end {proof}
%
\begin {clm} \label{minimizing direction}
 Suppose that for every $i \ne j$, $\sphericalangle (G_i,G_j) > \pi/3 $. There exist  $\alpha < \pi /2$ such that if $(x,y,z) \in T$ is a triangle whose angle at say $x$ is smaller than $\pi /3$, then there exist $W_0 \in  T_x^{G_{i,j}}$ with $\sphericalangle ([x,y],\exp(W_0)) < \alpha$ and $\sphericalangle ([x,z],\exp(W_0))<\alpha$.
\end {clm}
\begin {proof}
Let $c_1(t)\subset X^{G_i},c_2(t)\subset X^{G_j}$ be the geodesic paths issuing form $x$ to $y,z$ respectively. As above, also denote $V=c'_1(0)$ and $W=c'_2(0)$, and write $V=V_0+V_1$ where $V_0=P_{i,j} V$ is the orthogonal projection of $V$ on   $T_x^{\langle G_i,G_j\rangle}$  (use the same notation for $W$)

 The assumption $\sphericalangle (V,W) <\pi/3$ reads:
  $$\frac {\langle V,W \rangle} {\parallel V\parallel \parallel W\parallel } =\frac {\langle V_0,W_0 \rangle} {\parallel V\parallel \parallel W\parallel } +\frac {\langle V_1,W_1 \rangle} {\parallel V\parallel \parallel W\parallel }> \frac{1}{2}$$
By our assumption on the angles between $G_i$ and $G_j$, we have $\delta >0$ for which:
 $$ \frac {\langle V_1,W_1 \rangle} {\parallel V\parallel \parallel W\parallel }<\frac {\langle V_1,W_1 \rangle} {\parallel V_1\parallel \parallel W_1\parallel } < \frac {1}{2} - \delta ,$$
 hence
  $$\frac {\langle V_0,W_0 \rangle} {\parallel V\parallel \parallel W\parallel } > \delta$$
   which by Cauchy Schwartz inequality implies:
 $$\parallel V_0\parallel \parallel W_0\parallel >\delta \parallel V\parallel \parallel W\parallel .$$
Denote $w=w(t)=\exp(W_0) \subset X^{G_{i,j}}$.
 It follows then that the angle between $[x,y]$ and $w$ as well as the angle between $[x,z]$ and $w$ are bounded from above by $\alpha <\pi/2$. Indeed:
$$\frac {\langle W_0,W\rangle} {\parallel W_0\parallel \parallel W\parallel }=\frac {\langle W_0,W_0\rangle} {\parallel W_0\parallel \parallel W\parallel } = \frac {\parallel W_0\parallel }{\parallel W\parallel }>\delta$$
and also
$$ \frac {\langle W_0,V\rangle} {\parallel W_0\parallel \parallel V\parallel }> \delta.$$
\end {proof}
\begin {rem}\label{nearest point projection}
 Recall that if $C$ is a closed convex set in a CAT(0) space and $y \notin C$ is any point not in $C$, then there is a unique point denoted by $\pi_C(y)\in C$ such that $d(y,\pi_C(y))=d(y,C)$. The point $\pi_C (y)$ is called the nearest point projection of $y$ in $C$. Recall further that if $x\in C$ and $x \neq \pi_C(y)$, then the angle at $\pi_C(y)$ between $[\pi_C(y),x]$ and $[\pi_C(y),y]$ is $\geq \pi/2$, hence when $C$ is a geodesic line the angle at $\pi_C(y)$ is equal to $\pi/2$. It follows from convexity that any point on the segment $[x,\pi (y)]$ is closer to $y$ than $x$. For all this see Proposition II.2.4 in \cite{BH}.

\end{rem}
\begin {lem} \label {inf triangle}
Suppose that for every $i\ne j$, $\sphericalangle (G_i,G_j) >\pi/3 $. Then $\inf _{T}{f} =0$.
\end {lem}

\begin {proof}
Assume towards contradiction that $\inf _{T}{f} =L>0$. Fix $\epsilon>0$ and suppose that the triangle $(x,y,z) \in T$ has $f(x,y,z)<L+ \epsilon$. As the sum of angles in triangles in CAT(0) spaces cannot be greater than $\pi$, one of the angles, say at $x$ must be smaller than $\pi/3$. Let $w \subset X^{G_{1,2}}$ be a geodesic ray issuing from $x$ with $\sphericalangle ([x,y],w) < \alpha$ and $\sphericalangle ([x,z],w)<\alpha$ (with $\alpha < \pi/2$, see Claim \ref {minimizing direction}). Denote by $x'$ and $x''$ the nearest point projections of $y$  and $z$ on $w$ respectively (see Remark \ref{nearest point projection}). As $\sphericalangle ([x,y],w) < \alpha< \pi/2$, it follows from the same remark that $x' \ne x$. Similarly $x'' \ne x$. By the same reasoning (see the second part of Remark \ref{nearest point projection}) every point $s \in [x,x']$ has $d(s,y)<d(x,y)$ and similarly every $s \in [x'',x]$ has $d(x'',z)<d(x,z)$. Without loss of generality assume that $d(x',x)\le d(x'',x)$, so $d(x',y)<d(x,y)$ and $d(x',z)<d(x,z)$. We conclude then that  $L \le f(x',y,z)<f(x,y,z)$, hence $f(x,y,z)-f(x',y,z)<\epsilon$. It follows then that 
\[
d^2(x,y)-d^2(x',y)<\epsilon.
\]
On the other hand we have for the Euclidean comparison triangle $\overline {\Delta}(\overline{x},\overline {x'},\overline {y}) \subset \BE^2$, that the angle at $x'$ is $>\pi/2$. Thus 
\[
d^2(x',x) \leq d^2(x,y)-d^2(x',y).
\]
Combining the two we see that $d^2(x,x') \leq \epsilon$.
Note that we can assume that
 \[
d(x,y)>\sqrt{{L/10}}
\]
since otherwise by Claim \ref{switching vertices} the angle at $z$ is also smaller than $\pi/3$ and we could argue there.
To conclude, as $\epsilon$ tends to $0$ we observe triangles $\{x,x',y\}$ with the following properties:
\begin{enumerate}
\item $d(x,x') \leq \sqrt{ \epsilon}$.
\item $\sphericalangle([x,y],[x,x'] )<\alpha$ with $\alpha $ independent of $\epsilon$.
\item $\sphericalangle([x',y],[x',x] )=\pi/2$
\item $d(x',y) > L/10$
\end{enumerate}
This is impossible as the sectional curvature is bounded from below hence we get a contradiction.
\end {proof}

Our goal now is to prove Theorem \ref {infinite fix point criterion}. So far we showed that  $\inf _{T}{f} =0$. We are left to show that $f$ attains a minimum. To this end we define an auxiliary function $h$ as follows:

 \begin{defn}
Let $x \in X$ be any point and $C$ be a closed convex subset of $X$. As in Remark \ref {nearest point projection}, we denote by $\pi_C(x)$ the nearest point projection of $x$ in $C$. Now define:
 $$ h:X \rightarrow \BR, ~ x\mapsto d(x,\pi_{X^{\langle G_1,G_2\rangle}}(x))^2+d(x,\pi_{X^{\langle G_1,G_3\rangle}}(x))^2+d(x,\pi_{X^{\langle G_3,G_2\rangle}}(x))^2$$
\end{defn}
 One can easily observe that both $f$ and $h$ have zero infimum together namely:
 \begin {clm} \label {infimum of f}
 $\inf_{x \in X}{h(x)}=0$ iff $\inf_{\gD \in T}{f(\gD)}=0$.
 \end {clm}
 \begin {proof}
 Indeed for the if part consider the circumcenter of small triangle and apply the CAT(0) inequality. The only if part  follows from the triangle inequality .
 \end {proof}
The following is immediate:
\begin{cor} \label {infimum of h}
Suppose that for every $i\ne j$, $\sphericalangle (G_i,G_j) >\pi/3 $. Then $\inf_{x \in X}{h(x)}=0$.
\end{cor}
Let $K <G$ be a compact (i.e. finite) symmetric generating set of $G$ and $x$ any point in $X$. We define the diameter of $x$, $diam (K\cdot x)$:
\begin {defn}
$diam (K\cdot x)=\max_{k\in K}d(x,k\cdot x)$
\end {defn}
The main step in showing that $f$ attains minimum is to construct a limit space upon which $G$ acts fixed point freely. In order to ensure absence of a fixed point we will need to bound the diameter of points which are close to our base points. The next easy claim will help us in this task. It will enable us to replace "bad" points with "good" ones.
\begin{clm}\label{connection between h and diam}
Let $x\in X$ be any point in $X$. Let $K_1=\langle G_2,G_3 \rangle$, $K_2=\langle G_1,G_3 \rangle$, $K_3=\langle G_1,G_2 \rangle$, and $K=K_1 \cup K_2 \cup K_3$. then
 $$\frac{1}{4}diam (K\cdot x)^2\leq h(x) \leq 3diam(K\cdot x)^2.$$
\end{clm}

\begin {proof}
Suppose that $diam(K\cdot x)=d(x,gx)$. Without loss of generality we can assume that $g \in \langle G_1,G_2 \rangle =K_3$. 
Then by triangle inequality
$$d(x,\pi_{X^{\langle G_1,G_2\rangle}}(x)) +d(gx,\pi_{X^{\langle G_1,G_2\rangle}}(x)) \geq d(x,gx)$$
The action is by isometries and $\pi_{X^{\langle G_1,G_2\rangle}}(x)$ is fixed by $g$ hence this reads:

$$d(x,\pi_{X^{\langle G_1,G_2\rangle}}(x)) \geq \frac{1}{2} d(x,gx)=\frac {1}{2}diam(K\cdot x).$$
In particular
\begin {equation} 
h(x) \geq \frac{1}{4}diam(K\cdot x)^2 
\end {equation}
and the left hand side of inequality is proved.
\\For the second:
Let $c_i$ be the circumcenter of $conv(K_i \cdot x)$, i.e. $c_i$ is the unique point minimizing the radius of ball containing  the convex hull of $K_i \cdot x$. Then on the one hand (by definition of circumcenter)
\[
d(x,c_i)\leq diam(K_i \cdot x) \leq diam (K\cdot x)
\]
On the other hand $c_i$ is $K_i$ fixed hence
$$d(x,c_i)\geq d(x,\pi_{X^{K{_i}}}(x))$$	
hence
\[
h(x) \leq 3 diam(K\cdot x)^2 
\]
\end {proof}
We turn now to prove Theorem \ref {infinite fix point criterion}
\begin {proof} [Proof of Theorem \ref {infinite fix point criterion}]
Let $X$ be (possibly infinite dimensional) Hadamard manifold whose sectional curvature is bounded below by $\kappa$. Assume by contradiction, that $G$ is acting isometrically on $X$ without a fixed point. Let $h$ be defined as above. By by Corollary \ref {infimum of h}, $\inf_{x \in X}{h(x)}=0$ .


%
%
Assume then that $h(z_n)\leq \frac{1}{2^n}$ for some sequence $\{z_n\}_{n\in\BN} \in X$. Next we apply a limit process (compare with Lemma 3.1 in \cite {BFGM}). Continue along the following steps:

I. Our first step is to construct out of $\{z_n\}_{n\in\BN}$ another sequence, having diameter bounded from below for nearby points, yet having vanishing of $h$.
\begin {clm} \label {sparse sequence}
There exist a sequence $\{(x_n,k_n)\}_{n\in\BN}$ (where $x_n \in X$ and $k_n \in \BN$) with $h(x_n)\leq \frac {1}{2^{n+k_n}}$ and $diam (K\cdot y)\geq \frac {1}{5} diam (K\cdot x_n)$ for every $y \in B(x_n,\frac {1}{(n+k_n)^2})$.
\end {clm}

\begin {proof}
Fix $n$ and start with $z_n$. By the way we chose it, $h(z_n)<\frac {1}{2^n}$. If however it happens that $diam (K\cdot y)< \frac {1}{5} diam (K \cdot z_n)$ for some $y \in B(z_n, \frac{1}{n^2})$, then by Claim \ref{connection between h and diam},
\[
h(y)\leq 3diam(K \cdot y)^2 <\frac{3}{25}(K \cdot z_n)^2 <\frac{1}{2}h(z_n)
\]
so $h(y)<\frac {1}{2^{n+1}}$. Denote $y_1^n=y$. If again it happens that $diam (K\cdot y)\le \frac {1}{5} diam (K \cdot y_1^n)$ for some $y \in B(y_1^n, \frac{1}{(n+1)^2})$, then by Claim \ref{connection between h and diam} we have again that also $h(y)<\frac {1}{2}h(y_1^n)<\frac {1}{2^{n+2}}$. Continue with this process obtaining a sequence $y_k^n$ with $h(y_k^n)<\frac {1}{2^{n+k}}$. 
\begin {clm}
This process has to terminate after finitely many times with $y_{k_n}^n$ which we denote by $x_n$. 
\end {clm}
\begin {proof}
Indeed otherwise $\{y_k^n\}_{k=1}^\infty$ is Cauchy sequence since $d(y_k^n,y_{k+1}^n)<\frac {1}{(n+k)^2}$. Since $X$ is complete it has a limit which has to be a $G$ fixed point.

\end {proof}

By construction $x_n$ is the desired sequence.

\end {proof}
II. In the second step we construct a limit space. Let $X_n=\frac{1}{diam(K \cdot x_n)}X$ denote the Hadamard manifolds $X$ with new metric $d_n=\frac {1}{diam(K \cdot x_n)}d$. The pointed spaces $(X_n,x_n)$ have the following nice properties:
\begin {enumerate}
\item \label {a} The sectional curvature of $X_n$ is bounded from below by $\frac {\kappa}{ diam(K \cdot x_n)}$.
\item \label {aaa} The action of $G$ induces an isometric action on $X_n$. In order to distinguish between the diameter of a point in $X$ and the diameter in $X_n$ we denote $Diam_n(K \cdot x)=\max_{k \in K} d_n(x,k \cdot x)=\frac {diam(K \cdot x)}{diam(K \cdot x_n)}$. By definition $Diam_n(K\cdot x_n)=1$. Moreover for any sequence $y_n \in X_n$ for which $d_n(y_n,x_n)$ is bounded by some $L >0$, $Diam_n(K\cdot y_n) \le 2L+1$. This follows from the triangle inequality.
\item \label {bbb} On the other hand for every such $y_n$,  $Diam_n(K\cdot y_n) \ge \frac{1}{5}$ for every $n$ large enough. Indeed if $d_n(y_n,x_n)<L$ (for some $L>0$), then by definition of $d_n$, 
$$d(y_n,x_n)<diam(K \cdot x_n)L.$$
We have then, 
$$d(y_n,x_n)<2L\sqrt{h(x_n)}$$
 (by Claim \ref {connection between h and diam}). By choice of $x_n$, $d(y_n,x_n)<2L\sqrt  {\frac {1}{2^{n+k_n}}} $. Hence for large enough $n$, $y_n \in B(x_n,\frac {1}{(n+k_n)^2})$. By the choice of $x_n$ (see \ref {sparse sequence}) $diam(K \cdot y_n)\geq \frac{1}{5}diam(K \cdot x_n)$, hence $Diam_n(K\cdot y_n) \ge \frac{1}{5}$.

\end {enumerate}
Fix a non principal ultra-filter $\mathcal {U}$ and let $\mathcal H$ be the ultra-product of the pointed spaces $(X_n,x_n)$. Then $\mathcal H$ is a Hilbert space (see \ref {limit of manifolds}). Property (\ref{aaa}) allows us to use (\ref {action on product}) in order to define an isometric action on $\mathcal H$. This action is fixed point free by property (\ref {bbb}).

However it follows from Theorem 5.9 in \cite {EJ} as well as Theorem 1.2 in \cite {K} , that $G$ has property (T). By Delorme's Theorem $G$ then has also property FH (see for example Theorem 2.12.4 in \cite {BHV} and Theoreme V.1 in \cite {D} or the direct proof we gave \ref {FH}) hence we reached contradiction.



\end {proof}

The proof of Theorem \ref {infinite dimension} is now easy:
\begin {proof} [Proof of Theorem \ref {infinite dimension}.]
Theorem \ref {infinite dimension} follows from Theorem \ref {infinite fix point criterion} combined with  Lemma \ref {lem:>60}.
\end {proof}

We turn now to prove Theorem \ref {main result}. The proof relies on the well known fact that abelian groups that act on finite dimensional Hadamard manifolds without fixing any point must have element of infinite order. This fact follows from the fact that the fixed point set of any element is a complete Hadamard submanifold hence one can argue by induction.

\begin {proof} [Proof of Theorem \ref {main result}.]
Suppose that $\ga$ is an isometric $\gC$ action.    For fixed $1 \leq i \ne j \leq 3$, denote the abelian subgroup (isomorphic to the additive group of $R$),
  $$H_{i,j}= \{ x_{i,j}(r) ~s.t. ~ r \in R\} \cong R.$$
Then $H_{i,j}$ is an abelian group whose elements are of finite order, hence the restriction of $\ga$ to $H_{i,j}$ fixes a point $x \in X$. Suppose then that $x \in X$ is fixed by $H_{i,j}$. Since the action is by isometries, the image of a point $y$ is determined by the image of the geodesic segment $[x,y]$. The action is therefore fully determined by the images of geodesic rays issuing from $x$. Thus the study of isometric actions with fixed point is reduced to the study of the induced finite dimensional orthogonal representations on $T_x$. Let $\rho_{i,j}$ denote the finite dimensional orthogonal representation induced by the restriction of $\ga$ to $H_{i,j}$. Then  $\rho_{i,j}$ is a direct sum of one dimensional representations. Write:
  \[
   \rho_{i,j}   =\bigoplus_{k=1}^m \chi_k
  \]
 (with $\chi_k \in \hat{R}$ characters on $R$ and $m=dimX$).
  Observe that as $R$ is a direct sum of finite groups (namely copies $\BF_p$), its dual $\hat{R}$ is isomorphic to the product $\prod_{n \in \BN} \BF_p$.
    \begin {clm}
      Let $A_k=\ker \chi_k$. Then $A_k<H_{i,j}\cong R$ is subgroup of finite index.
     \end {clm}
    \begin {proof}
      Indeed as the range of $\chi_k$ has $p$ elements the kernel is of index $p$.
    \end {proof}
  \begin {cor} \label {the kernel is finite index ideal}
    Let $U_{i,j} =\{ r \in R | x_{i,j}(r) \in \ker \rho_{i,j}\} \cong  \bigcap_{k=1}^m A_k$. Further let $U=\bigcap_{1 \leq i \neq j \leq 3} U_{i,j}$ then $U$ is a finite index two sided ideal in  $R$.
  \end {cor}
\begin {proof}
Note first that for every $i,j$, if $u \in U$ then by definition $x_{i,j}(u)$ acts trivially. Now as it's a finite intersection of finite index subgroups it is also finite index and it is closed under addition. Suppose further that $u \in U$, take any $r \in R$ by the defining relations of the Steinberg group we obtain that $[x_{i,j}(r),x_{j,k}(u)]=x_{i,k}(ru)$. As $u \in U$ acts trivially the left hand side is also in the kernel, hence $U$ is closed under left multiplication by elements in $R$. Observe that $[x_{i,j}(u),x_{j,k}(r)]=x_{i,k}(ur)$. It follows that $U$ is a two sided ideal.   

\end {proof}

Next we adopt Milnor's notation (used in \cite {M}). We denote by $St_n(U)$ the normal closure of the group generated by elements of the form $x_{i,j}(u)$, $u \in U$. This group is generated by elements of the form $sx_{i,j}(u)s^{-1}$ (with $s \in St_n(\BF_p \langle t_1,\ldots ,t_k\rangle)$. It follows that $St_n(U) \lhd St_n(R)$ is in the kernel of $\ga$. By Lemma 6.1 in \cite {M} we obtain a short exact sequence
$$1 \rightarrow St_n(U) \rightarrow St_n(R) \rightarrow St_n(R/U) \rightarrow 1.$$
 On the other hand, Kassabov and Sapir showed that when $R/U$ is finite, $St_n(R/U)$ is finite also (Lemma 17 in \cite {KS}). This proves that the kernel of $\ga$ is of finite index and finishes the proof.

\end {proof}

 \subsection {The Torsion Free Case}
We now turn to deal with the groups $EL_n(R)$ when $R= \BZ \langle   t_1,\ldots,t_k\rangle$. As above our results will be slightly more general since we work with $St_n(R)$ instead. Note that although (similarly to the case $R= \BF_p \langle   t_1,\ldots,t_k\rangle$)   these groups are still generated by groups of the form
\begin{itemize}
\item
$G_i= \{x_{i,i+1}(a)\}$, where $a \in \BZ$,  $1 \leq i \leq n-1$, and
\item $G_n=\{x_{n,1}(a_0+a_1t_1+\ldots a_kt_k)\}$ , where $a_i \in \BZ$,  $1 \leq i \leq k$.

\end{itemize}
 these groups are infinite hence our method would fail in the first step. Indeed these groups can act by  hyperbolic isometries without fixing any point at all. However much is known about co-compact proper actions of solvable groups on CAT(0) spaces. Our main tool in proving \ref {non  commutative universal lattices} will be the solvable subgroup theorem which we will describe next. We begin by reminding the definition of a metrically proper action.
\begin {defn} \label {proper action}
Let $G$ be a discrete group acting on a CAT(0) space $X$ by isometries. We say that the action is metrically proper if for every $x \in X$ there is $r >0$ such that the set $\{ g \in G ~ s.t. ~ g.B(x,r) \cap B(x,r) \neq \emptyset \}$ is finite.
\end {defn}
Note that this definition is in general more restrictive than the usual definition of proper actions, which regards to compact sets in $X$. Even though one can clearly see both definitions coincide in the case of proper spaces (see definition I.8.2 and the following remark in \cite {BH}).
The solvable subgroup theorem states that if a group $\Gamma$ acts metrically properly and  co-compactly on CAT(0) space then any solvable subgroup $S< \Gamma$  is finitely generated and more importantly, it is virtually abelian. It is straightforward to deduce from it that non-uniform irreducible lattices of higher rank  semi simple Lie groups that have no compact factors cannot act metrically properly and co-compactly on CAT(0) spaces (see theorem II.7.8 and the following remark in \cite {BH}. Thus Theorem \ref {non  commutative universal lattices} is a generalization of this. Theorem \ref {non  commutative universal lattices} will follow easily from the the solvable subgroup theorem combined with the following rigidity property:
\begin{prop}\label{prop:algebraic rigidity}
Let $H$ be a topological group all of whose solvable subgroups are virtually abelian and $\gC$ an universal lattice, then any group homomorphism $\phi : \gC \rightarrow H$ has finite image.
\end{prop}
\begin {proof}
Suppose that we have a group homomorphism: $\phi : \gC \rightarrow H$. Similarly to the case studied  above we denote $G_{i,j}= \langle x_{i,i+1}(R), x_{j,j+1}(R) \rangle$. We study the image of the solvable (infinite Heisenberg) group $G_{1,2}$. Observe that by simple calculation the derived subgroup $[G_{1,2},G_{1,2}]$ is just the subgroup $E_{1,3}(R)$ of matrices with $1$ on the diagonal, elements of $R$ in the $(1,3)$ position and $0$ elsewhere. By the assumption regarding to solvable subgroups in $H$, the image of $G_{1,2}$ is virtually abelian hence $\ker \phi \cap [G_{1,2},G_{1,2}] $ is of finite index in $[G_{1,2},G_{1,2}]=E_{1,3} \cong R$.

For fixed $1 \leq i \ne j \leq n$ denote $A_{i,j}= \ker \phi \cap x_{i,j}(R)$. Denote further,  $U_{i,j} =\{ r \in R | x_{i,j}(r) \in A_{i,j}\}$.  We proceed in a similar manner to the end of the proof of Theorem \ref {main result}. First we show that for any $i,j$ we have $U_{i,j}=U_{1,3}$. Indeed Observe that  if $r \in U_{1,3}$ then $[x_{1,3}(r),x_{3,k}(1)] \in A_{1,k}$. Since $[x_{1,3}(r),x_{3,k}(1)]=x_{1,k}(r)$ it follows that $r \in U_{1,k}$ for any $k \ne 1$. But then (in a similar manner)  $[x_{k,1}(1),x_{1,j}(r)]=x_{k,j}(r)$ implies that $r \in U_{k,j}$ for any $k,j \ne 1$. Finally $[x_{k,j}(r),x_{j,1}(1)]=x_{k,1}(r)$ gives that $r \in U_{j,1}$ and $U_{1,3} \subset U_{i,j}$ (one gets the opposite inclusion similarly). 
 
The subrings $U_{i,j}$ are therefore independent of $i,j$ so we denote them by $U$. Next we show that $U<R$ is a finite index two sided ideal. Indeed by definition it is a finite index (additive) subgroup in $R$. Moreover if $r \in U$ and $s$ is any element  in $R$ then $[x_{1,3}(r),x_{3,k}(s)]=x_{1,k}(rs) \in A_{1,k}$ (hence also in $A_{i,j}$ for any $i \ne j$). Therefore $rs \in U$, and $U$ is closed under right multiplication by elements of $R$. Similarly it is also a left ideal.

 We obtain again a short exact sequence
 $$1 \rightarrow St_n(U) \rightarrow St_n(R) \rightarrow St_n(R/U) \rightarrow 1$$ 
and again use the fact that $St_n(R/U)$ is finite when $R/U$ is finite to deduce that the image of $\phi$ is finite.
\end {proof}

\begin {proof}[Proof of Theorem \ref {non  commutative universal lattices}.]
The theorem is easily deduced from Proposition \ref {prop:algebraic rigidity} combined with the solvable group theorem described above.
\end {proof}

\bibliography{cat}{}
\bibliographystyle{plain}

\end{document}